\documentclass[preprint,11pt]{elsarticle}

\usepackage{amssymb}

\usepackage{amsfonts,amsmath,latexsym,amssymb} 
\usepackage{amsthm}               
\usepackage{mathrsfs,upref}         
\usepackage{mathptmx}		    

\newtheorem{theorem}{Theorem}           
\newtheorem{lemma}{Lemma}               

\newtheorem{proposition}{Proposition}

\theoremstyle{definition}

\newtheorem{definition}{Definition}

\newtheorem{remark}{Remark}

\begin{document}

\begin{frontmatter}


\title{On Eigenvalue Problems Related to the Laplacian in a Class of Doubly Connected Domains
}


\author{Sheela Verma}
\ead{sheela.verma23@gmail.com}
\address{ Department of Mathematics and Statistics\\ Indian Institute of Technology Kanpur\\ Kanpur, India}

\author{G. Santhanam}
\ead{santhana@iitk.ac.in}
\address{ Department of Mathematics and Statistics\\ Indian Institute of Technology Kanpur\\ Kanpur, India}
\begin{abstract}
We study eigenvalue problems in some specific class of doubly connected domains. In particular, we prove the following.
\begin{enumerate}
\item Let $B_1$ be an open ball in $\mathbb{R}^n$, $n>2$ and $B_0$ be an open ball contained in $B_1$. Let  $\nu$ be the outward unit normal on $\partial B_1$. Then the first eigenvalue of the problem
\begin{align*}
\begin{array}{rcll}
\Delta u &=& 0 \, &\mbox{ in } \, B_1 \setminus \overline{B}_0  , \\
u &=& 0 \, &\mbox{ on } \, {\partial B_0}, \\
\frac{\partial u}{\partial \nu} &=& \tau u \, &\mbox{ on } \, {\partial B_1},
\end{array}
\end{align*}
attains its maximum if and only if $B_0$ and $B_1$ are concentric.

\item Let $(\mathbb{M}, ds^2)$ be a non-compact rank-$1$ symmetric space with curvature $-4 \leq K_{\mathbb{M}} \leq -1$ and $M = \mathbb{M}$ or $\mathbb{R}^{m}$. Let $B_0 \subset M$ be a geodesic ball of radius $r$ centered at a point $p \in M$. Let $D \subset M$ be a domain of fixed volume which is geodesically symmetric with respect to the point $ p\in M$ such that $\overline{{B}_0}\subset D$, and $\nu$ be the outward unit normal on ${\partial  (D \setminus \overline{B}_0)}$. Then the first non-zero eigenvalue of
\begin{align*}
\begin{array}{rcll}
\Delta u &=& \mu u \, &\mbox{ in } \,  D \setminus \overline{B}_0, \\
\frac{\partial u}{\partial \nu} &=& 0 \,  &\mbox{ on } \, {\partial  (D \setminus \overline{B}_0)},
\end{array}
\end{align*}
attains its maximum if and only if $D$ is a geodesic ball centered at $p$. 
\end{enumerate}

\end{abstract}


\end{frontmatter}

\section{Introduction}
The study of an eigenvalue problem on a punctured domain has been a topic of interest. Several interesting results have been proved in this area by considering various boundary conditions on a punctured domain. In \cite{ARPS}, Ramm and Shivakumar considered Dirichlet boundary condition on a punctured ball in $\mathbb{R}^2$(a ball of smaller radius is removed from a ball) and proved that the first eigenvalue of this problem  attains its maximum if and only if the balls are concentric. In \cite{SK}, Kesavan proved the above result for higher dimensions. Later this result was extended to a wider class of domains \cite{ARRR,ACMV,ACAR,SK1,SK2}. In \cite{RTIB}, Banuelos et al. proved some classical inequalities between the eigenvalues of the mixed Steklov-Dirichlet and mixed Steklov-Neumann eigenvalue problems on a bounded domain in $\mathbb{R}^n$.

In this paper, we consider the mixed Steklov-Dirichlet problem and Neumann eigenvalue problem on some specific class of punctured domains and prove that the first non-zero eigenvalue of both problems is maximal for annular domain. In between, we also give the characterization of the first non-zero eigenvalue of the Neumann problem on an annulus. Now we state our main results.

For $n>2$, let $B_1$ be an open ball in $\mathbb{R}^n$ of radius $R_1$ and $B_{2}'$ be an open ball in $\mathbb{R}^n$ of radius $R_2$ such that $\overline{B}_1\subset B_2'$, where $0 < R_1 <R_2$. Consider the class $\mathcal{O}_{R_{1}, R_{2}} = \left\lbrace B_{2}' \setminus \overline{B}_1 \right\rbrace $ of domains in $\mathbb{R}^n$. Then we prove the following result.

\begin{theorem} \label{thm:steklov}
Consider the eigenvalue problem
\begin{align} \label{eqn:steklov}
\begin{array}{rcll}
\Delta u &=& 0 \, &\mbox{ in } \, B_{2}' \setminus \overline{B}_1  , \\
u &=& 0 \, &\mbox{ on } \, {\partial B_1}, \\
\frac{\partial u}{\partial \nu} &=& \tau u \, &\mbox{ on } \, {\partial B_{2}'},
\end{array}
\end{align}
where $\nu$ is the outward unit normal on $\partial B_{2}'$. Then annular domains (concentric balls) maximize the first eigenvalue of \eqref{eqn:steklov} in the class $\mathcal{O}_{R_{1}, R_{2}}$. 
\end{theorem}

\begin{definition} A domain $D$ in a Riemannian manifold $M$ is said to be \textit{geodesically symmetric} with respect to a point $p \in M$ if $D = \exp_{p} (N_0)$, where $N_0$ is a symmetric neighborhood of the origin in $T_{p}M$.
\end{definition}
 We now state our result for the first non-zero eigenvalue of Neumann problem on doubly connected domains in $\mathbb{R}^m$ and non-compact rank-$1$ symmetric space $\mathbb{M}$ with curvature $-4 \leq K_{\mathbb{M}} \leq -1$.
 
Let $(\mathbb{M}, ds^2)$ be a non-compact rank-$1$ symmetric space with dim $\mathbb{M} = m = kn$, where $k= \mbox{dim}_{\mathbb{R}} {\mathbb{K}}$; $\mathbb{K} = \mathbb{R}, \mathbb{C}, \mathbb{H}$ or $\mathbb{C}a$, and metric $ds^{2}$ is such that $-4 \leq K_{\mathbb{M}} \leq -1$. Let $M = \mathbb{M}$ or $\mathbb{R}^m$. Let $B_0 \subset M$ be a geodesic ball of radius $r_1$ centered at a point $p \in M$, and $D \subset M$ be a domain of fixed volume which is geodesically symmetric with respect to the point $p$ such that $\overline{B}_0 \subset D $. We consider the following eigenvalue problem on a domain $\Omega$ in the class $\mathcal{F}= \left\lbrace D \setminus \overline{B}_0\right\rbrace $. 
\begin{align} \label{eqn:neumann}
\begin{array}{rcll}
\Delta u &=& \mu u \, &\mbox{ in } \, \Omega, \\
\frac{\partial u}{\partial \nu} &=& 0 \,  &\mbox{ on } \, {\partial \Omega},
\end{array}
\end{align}
where $\nu$ is the outward unit normal on $\partial \Omega$.

\begin{theorem} \label{thm:neumann}
Among all domains $\Omega \subset M $ in the class $\mathcal{F}$, annular domain(concentric ball) maximizes the first non-zero eigenvalue of \eqref{eqn:neumann}.
\end{theorem}

The key idea to prove above results is to construct a suitable test function for the variational characterization of first non-zero eigenvalue for the corresponding eigenvalue problem.

This article is organized as follows. In Section 2, we consider the mixed Steklov-Dirichlet problem on an annulus and calculate a first eigenfunction by separation of variable technique. Then we prove some inequalities which are important to prove Theorem \ref{thm:steklov} and then provide a proof of Theorem \ref{thm:steklov}. In Section 3, we begin with the study of Neumann eigenvalue problem on annulus and then prove Theorem \ref{thm:neumann}.

\section{The mixed Steklov-Dirichlet problem} We begin with the study of mixed Steklov-Dirichlet eigenvalue problem on an annular domain. 

\subsection{The mixed Steklov-Dirichlet problem on an annular domain}
Let $0 < R_1 < R_2$ and $\Omega_{0} = \left\lbrace x \in \mathbb{R}^{n}\mid R_1 < \Vert x \Vert < R_2 \right\rbrace $ be an annular domain in $\mathbb{R}^n$. Let $B_1$ and $B_{2}$ be open balls in $\mathbb{R}^n$ centered at the origin and of radius $R_1$ and $R_2$, respectively. Consider the eigenvalue problem 
\begin{align}  \label{steklov on annular}
\begin{array}{rcll}
\Delta u &=& 0 \, &\mbox{ in } \, \Omega_{0}  , \\
u &=& 0 \, &\mbox{ on } \, {\partial B_1}, \\
\frac{\partial u}{\partial r} &=& \tau u \, &\mbox{ on } \, {\partial B_2}.
\end{array}
\end{align}
Let $\tau_{1} (\Omega_{0})$ denote the first eigenvalue of \eqref{steklov on annular}.
We compute the first eigenvalue and corresponding eigenfunction of \eqref{steklov on annular} by using the separation of variable technique. 

Consider a smooth function on $\Omega_{0}$ given by $h(r,u) = f(r) g(u)$, where $f$ and $g$ are real valued functions defined on $[R_{1}, R_{2}]$ and $\mathbb{S}^{n-1}$, respectively. Let $S(r)$ be the sphere of radius $r$ centered at the origin and $g(u)$ is an eigenfunction of $\Delta_{S(r)}$ with eigenvalue $\lambda (S(r))$. Then
\begin{align*}
\Delta h(r,u) &= g(u) \left[ -  f'' (r)- \frac{(n-1)}{r}  f ' (r)\right]  +  f  (r) \Delta_{S(r)} g(u) \\
 &=  g(u) \left[-  f'' (r)- \frac{(n-1)}{r}  f ' (r) + \lambda (S(r)) f  (r) \right].  
\end{align*}
Assume that $h(r,u)$ is a solution of \eqref{steklov on annular}. Then the function $f$ satisfies
\begin{align} \label{efun and evalue for steklov}
\begin{array}{rcll}
-  f'' (r)- \frac{(n-1)}{r}  f ' (r) + \lambda (S(r)) f  (r)=0, \\ 
 f (R_1) =0, \, f' (R_2) = \tau f (R_2).
 \end{array}
\end{align}

Let $0 = \lambda_{0} (S(r)) < \lambda_{1} (S(r)) \leq \lambda_{2} (S(r))\ldots \nearrow \infty $ be the eigenvalues of $\Delta_{S(r)}$. We will now prove that if $\lambda (S(r)) = \lambda_{0} (S(r))$ and $\tau_{0}$ is the corresponding eigenvalue of \eqref{efun and evalue for steklov}, then $\tau_{0}$ is the first eigenvalue of \eqref{steklov on annular}.

For $i \geq 0$, let $f_i$ be functions satisfying
\begin{align}\label{sep variable for steklov}
\begin{array}{rcll}
-  f_{i}'' (r)- \frac{(n-1)}{r}  f_{i} ' (r) + \lambda_i (S(r)) f_{i}  (r)=0, \\ 
 f_i (R_1) =0, \, f_{i}' (R_2) = \tau_i f_i (R_2),
 \end{array}
\end{align}
where $\tau_{i}$ is the first eigenvalue of \eqref{sep variable for steklov} and $f_{i}$ is an eigenfunction corresponds to $\tau_{i}$. Thus $f_i$ will not change sign in $[R_1, R_2]$. We may assume that $f_i$'s are positive functions for  $i \geq 0$. 
The above equation is equivalent to 
\begin{align*} 
r \,  f_{i} '' (r)+ (n-1)\,  f_{i} ' (r) - r \, \lambda_i (S(r)) \, f_i (r)=0, \\ 
f_i (R_1) =0, \, f_{i}' (R_2) = \tau_i f_i (R_2).
\end{align*}
We will now prove that $\tau_i \leq \tau_{i+1}$ for  $i \geq 0$.

Let's fix $i$. Multiply the equation
$$
r \,  f_{i} '' + (n-1)\,  f_{i} '  - r \, \lambda_i (S(r)) \, f_i =0
$$
by $f_{i+1}$ and the equation
$$
r \,  f_{i+1} '' + (n-1)\,  f_{i+1} '  - r \, \lambda_{i+1} (S(r)) \, f_{i+1} =0
$$
by $f_i$. By subtracting the second equation from the first and using the fact $\lambda_{i} (S(r)) \leq \lambda_{i+1} (S(r))$, we get
\begin{align*}
0 =& r^{n-2} \left(  f_{i+1}  \left(r \,  f_{i} '' + (n-1)\,  f_{i} '  - r \, \lambda_i (S(r)) \, f_i  \right) - f_{i}  \left( r \,  f_{i+1} '' + (n-1)\,  f_{i+1} '  - r \, \lambda_{i+1} (S(r)) \, f_{i+1}  \right)\right)   \\
 =&  r^{n-2} \left( r \, \left(  f_{i} '' \, f_{i+1} -   f_{i+1} '' \, f_{i} \right) + (n-1)\,  \left(  f_{i} ' \, f_{i+1} -   f_{i+1} ' \, f_{i} \right) - r f_i \, f_{i+1} \left( \lambda_{i} (S(r))  - \lambda_{i+1} (S(r)) \right)\right)  \\
 \geq & \, r^{n-1} \, \left(  f_{i} '' \, f_{i+1} -   f_{i+1} '' \, f_{i} \right) + (n-1) \, r^{n-2} \,  \left(  f_{i} ' \, f_{i+1} -   f_{i+1} ' \, f_{i} \right)  \\
 = & \, \left( r^{n-1} \, \left(  f_{i} ' \, f_{i+1} -   f_{i+1} ' \, f_{i} \right)\right) ' .
\end{align*}

The condition $\left( r^{n-1} \, \left(  f_{i} ' \, f_{i+1} -   f_{i+1} ' \, f_{i} \right)\right) ' \leq 0 $ shows that $\left( r^{n-1} \, \left(  f_{i} ' \, f_{i+1} -   f_{i+1} ' \, f_{i} \right)\right)$ is a decreasing function of $r$. As a result,
\begin{align*}
R_{2}^{n-1} \, \left(  f_{i} ' (R_{2}) \, f_{i+1} (R_{2}) -   f_{i+1} ' (R_{2}) \, f_{i} (R_{2}) \right) \leq R_{1}^{n-1} \, \left(  f_{i} ' (R_{1}) \, f_{i+1} (R_{1}) -   f_{i+1} ' (R_{1}) \, f_{i} (R_{1}) \right).
\end{align*}
 Since $ f_{i} (R_{1}) = f_{i+1} (R_{1})=0$,  we have  $
R_{2}^{n-1} \, \left(  f_{i} ' (R_{2}) \, f_{i+1} (R_{2}) -   f_{i+1} ' (R_{2}) \, f_{i} (R_{2}) \right) \leq 0
$. This gives $f_{i} ' (R_{2}) \, f_{i+1} (R_{2}) \leq  f_{i+1} ' (R_{2}) \, f_{i} (R_{2})$. As a consequence,  we have
\begin{align*}
\tau_{i} = \frac{f_{i} ' (R_{2})}{f_{i} (R_{2})} \leq \frac{f_{i+1} ' (R_{2})}{f_{i+1}  (R_{2})} = \tau_{i+1}.
\end{align*}
Thus $\tau_{0}$ is the smallest eigenvalue of \eqref{efun and evalue for steklov} with eigenfunction $f_0$ and $\lambda (S(r)) = \lambda_{0} (S(r))$. Hence $ \tau_{1} (\Omega_{0}) = \tau_0$.

Next we compute the explicit form of $f_0$. Since $\lambda_{0} (S(r))=0$, $f_0$ satisfies the equation
$
r \,  f_{0} '' (r)+ (n-1)\,  f_{0} ' (r) =0
$.
This can also be written as 
$
\left( r^{n-1} \, f_{0}' (r) \right)' = 0.
$
Using the fact $f_{0} (R_1) = 0$, we obtain 
$$
f_{0}(r) = \int_{R_1} ^{r} \frac{c}{r^{n-1}} \, dr
$$
for some constant $c$.

\begin{remark}
Since constant function is an eigenfunction of $\Delta_{S(r)}$ with eigenvalue zero, without loss of generality, we may assume it to be $1$. Then the function $f_0$, the first eigenfunction of the eigenvalue problem \eqref{steklov on annular}, is given by
\begin{align}
f_{0} (r)=
\begin{cases}
\ln r - \ln R_{1}  ,  & \text{for } \, n=2, \\
\left( \frac{1}{R_{1}^ {n-2}} - \frac{1}{r^ {n-2}} \right) ,  & \text{for } \, n \geq 3.
 \end{cases}
\end{align}
Note that the first eigenfunction of \eqref{steklov on annular} is radial, positive and increasing function on $\Omega_{0}.$
\end{remark}

\subsection{Computations towards the proof of Theorem \ref{thm:steklov}}
In this section, we first prove some results needed to prove Theorem \ref{thm:steklov} and then present a proof of the theorem.

We fix some notations used in this subsection. To prove Theorem \ref{thm:steklov}, without loss of generality, we may assume that $B_1$ is an open ball in $\mathbb{R}^n$ of radius $R_1$ centered at the origin, and $B_{2} '$ is an open ball of radius $R_2$ such that $0 < R_1 < R_2$ and $\overline{B}_1 \subseteq B_{2}'$. In particular, we consider all domains $D= B_{2} ' \setminus \overline{B_1}$ by fixing the ball $B_1$ and shifting the outer ball such that $\overline{B}_1 \subseteq B_{2} '$. 
Then the variational characterization for the first eigenvalue of \eqref{eqn:steklov} is given by
\begin{align} \label{var char for steklov}
\tau_1 = \inf \left\lbrace \frac{\int_{D}{\|\nabla u\|^2}dV}{\int_{\partial B_{2} '}{u^2}dv}  \bigm|   u = 0 \text{ on } \partial B_1\right\rbrace.
\end{align}

Since the Laplacian is invariant under the rotations and reflections, if the outer ball moves so that its centre lies on a sphere centered at the origin, the first eigenvalue of \eqref{eqn:steklov} must remain same. Thus it is enough to consider that the centre of outer ball moves along the $x$-axis.
Let $B_{2}'$ and $B_2$ be the balls of radius $R_2$ in $\mathbb{R}^n$ centered at $\left( x,0,0,\ldots,0\right) $ and the  origin, respectively and $0 < x \leq (R_2-R_1)$. 
We denote $\partial B_{2}'$ by $\mathbb{S}'(R_2)$. The co-ordinates of $\mathbb{S}'(R_2)$ are given by \\
$x_1 = R \cos (\phi_1)+ x$, 
$x_2 = R \sin (\phi_1) \cos (\phi_2) $,
$x_3 = R \sin (\phi_1) \sin (\phi_2) \cos (\phi_3) $,
$\cdots$ ,
$x_n = R \sin (\phi_1) \sin (\phi_2)$ $ \cdots \sin (\phi_{n-1})$, where $ \phi_1, \phi_2,\ldots,\phi_{n-2} \in [0, \pi] $ and $\phi_{n-1} \in [0, 2 \pi] $.

Let $p$ be a point on $\mathbb{S}'(R_2)$ and $r$ be the distance of $p$  from the origin. Then $r= \sqrt{R^2 + x^2+ 2 \, R \, x \cos(\phi_1)}$.
\begin{definition} For $a, b, c, z \in \mathbb{C}$, \textit{the hypergeometric function} ${_2\mbox{F}_1}\left( a, b; c; z\right)  $ is defined as follows.
\begin{enumerate}

\item \textbf{Power series form:} If $c$ is not a non-positive integer, and either $\vert z \vert <1$ or  $\vert z \vert =1$ with Re$\left( c-a-b\right)> 0. $ Then
\begin{align*}
{_2\mbox{F}_1}\left( a, b; c; z \right) := {\sum_{m=0}^{\infty}} \frac{\left(a\right) _m \, \left( b \right)_m }{\left(c \right)_m  }  \frac{{z}^m}{m!},
\end{align*}
where $(q)_m$ is the (rising) Pochhammer symbol defined by
\begin{align*}
(q)_m =
\left\{
\begin{array}{lcll}
1 & m=0, \\
q\, (q+1) \cdots (q+m-1) & m>0.
\end{array}
\right.
\end{align*} 

\item \textbf{Integral representation:} If $\mbox{Re}(c) > \mbox{Re}(b) > 0$ and $z$ is not a real number greater than or equal to $1$.
\begin{align*}
{_2\mbox{F}_1}\left( a, b; c; z \right) := \frac{\Gamma\left(  c\right) }{\Gamma \left( b\right) \, \Gamma \left( c-b\right) } \int_{0}^{1} {x^{b-1} \, (1-x)^{c-b-1} \, (1-zx)^{-a}}  dx,
\end{align*}
where $\Gamma(s) =\int_{0}^{\infty} e ^{-t} \, t^{s-1}\, dt$. 
\end{enumerate}
\end{definition}

We use MATHEMATICA \cite{matica} in Lemma \ref{lem:odd1} and Lemma \ref{lem:odd2} to compute some integrals and sum of some series.
\begin{lemma} \label{lem:odd1}
Let $x$ and $R_{2}$ be as above. Then for all $k \in \mathbb{N}$,
\begin{align}
\int_{-1}^{1} \left( \frac{\left( 1-s^2\right)^{k-1}}{\left( R_{2}^2 + x^2+ 2 \, R_{2} \, x \, s\right) ^{2k-1}}   \right)  d s \geq
 \left({\frac{2k-2}{2k-1}} \, {\frac{2k-4}{2k-3}} \, \cdots {\frac{4}{3}} \right) \left(\frac{1}{R_2^{4k-2}}\right).
\end{align}
\end{lemma}
\begin{proof} By using MATHEMATICA, we get
\begin{align} \nonumber
\int_{-1}^{1} \left( \frac{\left( 1-s^2\right)^{k-1}}{\left( R_{2}^{2} + x^2+ 2 \, R_{2} \, x \, s\right) ^{2k-1}}   \right)  d s &= \frac{\sqrt{\pi}}{(R_{2}^2 + x^2)^{2k-1}} \, \frac{\Gamma\left(k \right)}{\Gamma\left( k + \frac{1}{2} \right)} \, {{_2\mbox{F}_1} \left(k - \frac{1}{2},k;\,k + \frac{1}{2};\,\frac{4 R_{2}^2 {x}^2}{(R_{2}^2 + x^2)^2} \right) }.
\end{align}

By substituting power series form of ${_2\mbox{F}_1} \left(k - \frac{1}{2},k;\,k + \frac{1}{2};\,\frac{4 R_{2}^2 {x}^2}{(R_{2}^2 + x^2)^2} \right)$ in the above expression, we obtain
\begin{align} \nonumber
& \int_{-1}^{1} \left( \frac{\left( 1-s^2\right)^{k-1}}{\left( R_{2}^{2} + x^2+ 2 \, R_{2} \, x \, s\right) ^{2k-1}}   \right)  d s \\ 
 \nonumber
&= \frac{\sqrt{\pi}}{(R_{2}^2 + x^2)^{2k-1}} \, \frac{\Gamma\left(k \right)}{\Gamma\left( k + \frac{1}{2} \right)} {\sum_{m=0}^{\infty}} \frac{\left( k - \frac{1}{2} \right)_m \, \left( k \right)_m }{\left(k + \frac{1}{2} \right)_m \, m! } \left( \frac{4 R_{2}^2 {x}^2}{(R_{2}^2 + x^2)^2} \right)^m \\ 
\nonumber
\nonumber
&= \frac{ \sqrt{\pi} \, \Gamma(k)}{\Gamma\left( k + \frac{1}{2} \right)} {\sum_{m=0}^{\infty}}\, \frac{ \Gamma \left( k +m \right) \,   \left(k - \frac{1}{2} \right) }{ \Gamma \left( k \right) \, \left(k - \frac{1}{2} +m \right) \, m! } \, \frac{\left( 4 \, {x}^2\right) ^m}{R_{2}^{2m+4k-2}} \left(1+ \frac{x^2}{R_{2}^2} \right)^{-(2m+2k-1)}  \\ 
\label{eqn:lem221}
&= \frac{ \sqrt{\pi} \, \Gamma(k)}{ R_{2}^{4k-2} \, \Gamma\left( k + \frac{1}{2} \right)} {\sum_{m,l=0}^{\infty}}\, \frac{ 4^{m} \, \Gamma \left( k +m \right) \,   \left(k - \frac{1}{2} \right) \, (-1)^l }{ \Gamma \left( k \right) \, \left(k - \frac{1}{2} +m \right) \, m! \, l!} \frac{\Gamma (2m + 2k-1+l)}{\Gamma (2m + 2k-1)} \left( \frac{x}{R_{2}}\right)^{2l+2m} .
\end{align}
Then the coefficient of $\frac{x^{2t}}{R_{2}^{2t+4k-2}}$ for $t \geq 0$ is given by
\begin{align*}  
\frac{ \sqrt{\pi} \, \Gamma(k)}{ \Gamma\left( k + \frac{1}{2} \right)} \,  {\sum_{i=0}^{t}} \, \frac{ 4^{i} \, \Gamma \left( k +i \right) \,   \left(k - \frac{1}{2} \right) \, (-1)^{t-i} }{ \Gamma \left( k \right) \, \left(k - \frac{1}{2} +i \right) \, i! \, (t-i)!} \, \frac{\Gamma (i + 2k-1+t)}{\Gamma (2i + 2k-1)}.
\end{align*}
Using MATHEMATICA, we have
\begin{align*}  
 {\sum_{i=0}^{t}} \, \frac{ 4^{i} \, \Gamma \left( k +i \right) \,   \left(k - \frac{1}{2} \right) \, (-1)^{t-i} }{ \Gamma \left( k \right) \, \left(k - \frac{1}{2} +i \right) \, i! \, (t-i)!} \, \frac{\Gamma (i + 2k-1+t)}{\Gamma (2i + 2k-1)} = \frac{(2k-1) \, \Gamma \left(2k +t-1 \right) }{(2k+2t-1)\, \Gamma \left( 2k-1 \right) \Gamma\left(t+1 \right)}.
\end{align*}
We denote the coefficient of $\frac{x^{2t}}{R_{2}^{2t+4k-2}}$ by $\alpha(t)$. Then \eqref{eqn:lem221} can be written as
\begin{align*}
\int_{-1}^{1} \left( \frac{\left( 1-s^2\right)^{k-1}}{\left( R_{2}^{2} + x^2+ 2 \, R_{2} \, x \, s\right) ^{2k-1}}   \right)  d s =  {\sum_{t=0}^{\infty}} \alpha(t) \, \frac{x^{2t}}{R_{2}^{2t+4k-2}}.
\end{align*}
 Note that $x, R_{2} > 0$ and $\alpha(t)$ is positive for all $t\geq0$. Thus
 \begin{align*}
\int_{-1}^{1} \left( \frac{\left( 1-s^2\right)^{k-1}}{\left( R_{2}^2 + x^2+ 2 \, R_{2} \, x \, s\right) ^{2k-1}}   \right)  ds \geq
 \alpha(0) \left(\frac{1}{R_2^{4k-2}}\right),
\end{align*}
where 
\begin{align*} 
\alpha(0) &= \frac{ \sqrt{\pi} \, \Gamma(k)}{ \Gamma\left( k + \frac{1}{2} \right)} \, \frac{(2k-1) \, \Gamma \left(2k -1 \right) }{(2k-1)\, \Gamma \left( 2k-1 \right) }\\ 
&= \left( \frac{2k-2}{2k-1} \, \frac{2k-4}{2k-3} \, \cdots \,  \frac{4}{3} \right) .
\end{align*}
This proves the lemma. 
\end{proof}

\begin{lemma} \label{lem:odd2}
Let $x$ and $R_{2}$ be as above. Then for all $k \in \mathbb{N}$,
\begin{align} \label{eqn: lem odd 1}
\int_{-1}^{1} \left( \frac{\left( 1-s^2\right)^{k-1}}{\left( R_{2}^2 + x^2+ 2 \, R_{2} \, x \, s\right) ^{\frac{2k-1}{2}}}   \right)  d s = \left({\frac{2k-2}{2k-1}} \, {\frac{2k-4}{2k-3}} \, \cdots {\frac{4}{3}} \right) \left(\frac{1}{R_2^{2k-1}}\right).
\end{align}
\end{lemma}
\begin{proof} By using MATHEMATICA, we get
\begin{align} \nonumber
\int_{-1}^{1} \left( \frac{\left( 1-s^2\right)^{k-1}}{\left( R_{2}^{2} + x^2+ 2 \, R_{2} \, x \, s\right) ^{\frac{2k-1}{2}}}   \right)  ds &= \frac{\sqrt{\pi}}{(R_{2}^2 + x^2)^{\frac{2k-1}{2}}} \, \frac{\Gamma\left(k \right)}{\Gamma\left( k + \frac{1}{2} \right)} \, {{_2\mbox{F}_1} \left(\frac{2k-1}{4},\frac{2k+1}{4};\,k + \frac{1}{2};\,\frac{4 R_{2}^2 {x}^2}{(R_{2}^2 + x^2)^2} \right) } .
\end{align}
By substituting the power series for the hypergeometric function in the above expression, we get
\begin{align}\nonumber
& \int_{-1}^{1} \left( \frac{\left( 1-s^2\right)^{k-1}}{\left( R_{2}^{2} + x^2+ 2 \, R_{2} \, x \, s\right) ^{\frac{2k-1}{2}}}   \right)  ds \\ \nonumber
&= \frac{\sqrt{\pi}}{(R_{2}^2 + x^2)^{\frac{2k-1}{2}}} \, \frac{\Gamma\left(k \right)}{\Gamma\left( k + \frac{1}{2} \right)} {\sum_{m=0}^{\infty}} \frac{\left( \frac{2k-1}{4} \right)_m \, \left( \frac{2k+1}{4} \right)_m }{\left(k + \frac{1}{2} \right)_m \, m! } \left( \frac{4 R_{2}^2 {x}^2}{(R_{2}^2 + x^2)^2} \right)^m \\ \nonumber
&= \frac{ \sqrt{\pi} \, \Gamma(k)}{\Gamma\left( k + \frac{1}{2} \right)} {\sum_{m=0}^{\infty}} \frac{\Gamma\left( \frac{2k-1}{4} +m \right) \, \Gamma \left( \frac{2k+1}{4} +m \right) \,  \Gamma \left(k + \frac{1}{2} \right) }{\Gamma\left( \frac{2k-1}{4} \right) \, \Gamma \left( \frac{2k+1}{4} \right) \, \Gamma \left(k + \frac{1}{2} +m \right) \, m! } \frac{\left( 4 R_{2}^2 {x}^2\right) ^m}{(R_{2}^2 + x^2)^{2m+ \frac{2k-1}{2}}} \\ \nonumber
&= \frac{ \sqrt{\pi} \, \Gamma(k)}{\Gamma\left( k + \frac{1}{2} \right)} {\sum_{m=0}^{\infty}} \frac{\Gamma\left( \frac{2k-1}{4} +m \right) \, \Gamma \left( \frac{2k+1}{4} +m \right) \,  \Gamma \left(k + \frac{1}{2} \right) }{\Gamma\left( \frac{2k-1}{4} \right) \, \Gamma \left( \frac{2k+1}{4} \right) \, \Gamma \left(k + \frac{1}{2} +m \right) \, m! } \frac{\left( 4 {x}^2\right) ^m}{R_{2}^{2m+ 2k-1}} \left( 1 + \frac{x^2}{R_{2}^2}\right)^{-\frac{4m + 2k - 1}{2}}  \\ \nonumber
&=\frac{ \sqrt{\pi} \, \Gamma(k)}{ R_{2}^{2k-1} \, \Gamma\left( k + \frac{1}{2} \right)} {\sum_{m,l=0}^{\infty}}\, \frac{ 4^{m} \, \Gamma\left( \frac{2k-1}{4} +m \right) \, \Gamma \left( \frac{2k+1}{4} +m \right) \,  \Gamma \left(k + \frac{1}{2} \right) \, (-1)^l }{ \Gamma\left( \frac{2k-1}{4} \right) \, \Gamma \left( \frac{2k+1}{4} \right) \, \Gamma \left(k + \frac{1}{2} +m \right) \, m!    \, l!} \frac{\Gamma (\frac{4m + 2k - 1}{2}+l)}{\Gamma (\frac{4m + 2k - 1}{2})} \left( \frac{x}{R_{2}}\right)^{2l+2m} \\ \label{eqn: lem odd 2}
 &= {\sum_{t=0}^{\infty}} \beta(t) \frac{x^{2t}}{R_{2}^{2t+2k-1}},
\end{align} 
where $\beta(t)$, $t \geq 0$, is given by
\begin{align*} 
\beta(t) = \frac{ \sqrt{\pi} \, \Gamma(k)}{ \Gamma\left( k + \frac{1}{2} \right)} {\sum_{i=0}^{t}}\, \frac{ 4^{i} \, \Gamma\left( \frac{2k-1}{4} +i \right) \, \Gamma \left( \frac{2k+1}{4} +i \right) \,  \Gamma \left(k + \frac{1}{2} \right) \, (-1)^{t-i} }{ \Gamma\left( \frac{2k-1}{4} \right) \, \Gamma \left( \frac{2k+1}{4} \right) \, \Gamma \left(k + \frac{1}{2} +i \right) \, i!    \, (t-i)!} \frac{\Gamma (\frac{4i + 2k - 1}{2}+t-i)}{\Gamma (\frac{4i + 2k - 1}{2})}.
\end{align*}
We simplify the above expression using MATHEMATICA and get
\begin{align*} 
\beta(t)=
\begin{cases}
\frac{ \sqrt{\pi} \, \Gamma(k)}{ \Gamma\left( k + \frac{1}{2} \right)}, & \text{ for } t=0, \\
0, & \text{ for } t \neq 0.
\end{cases}
\end{align*}
Hence
\begin{align*}
 \int_{-1}^{1} \left( \frac{\left( 1-s^2\right)^{k-1}}{\left( R_{2}^{2} + x^2+ 2 \, R_{2} \, x \, s\right) ^{\frac{2k-1}{2}}}   \right)  ds &= \beta(0) \frac{1}{R_{2}^{2k-1}} \\
& = \left({\frac{2k-2}{2k-1}} \, {\frac{2k-4}{2k-3}} \, \cdots {\frac{4}{3}} \right) \left(\frac{1}{R_2^{2k-1}}\right).  
\end{align*} 
\end{proof}

We denote the area element  $\sin^{2k-1} (\phi_1)\, \sin^{2k-2} (\phi_2)\, \cdots \, \sin (\phi_{2k-1}) \, d \phi_1 \, d \phi_2 \,\cdots \, d \phi_{2k}$ of $\mathbb{S}'(R_2)$ by $d\phi$ in the following proposition.

\begin{proposition} \label{prop: odd fun inequality}
Let $R_1$ and $R_2$ be defined as above and $r(p)$ denotes the distance of a point $p$ in $\mathbb{S}'(R_2)$ from the origin. For $n=2k+1, k \geq 1$, the following holds.
\begin{align}
\label{eqn:odd fun inequality}
\int_{\mathbb{S}'(R_2)}\left(\frac{1}{R_{1}^{2k-1}} - \frac{1}{{r}^{2k-1}} \right)^2 dv \geq \int_{\mathbb{S}(R_2)}\left(\frac{1}{R_{1}^{2k-1}} - \frac{1}{{R_2}^{2k-1}} \right)^2 dv.
\end{align}
\end{proposition}
\begin{proof}By substituting $r= \sqrt{R^2 + x^2+ 2 \, R \, x \cos(\phi_1)}$ and using the polar coordinate system, we have 
\begin{align*}
&\int_{\mathbb{S}'(R_2)}\left(\frac{1}{R_{1}^{2k-1}} - \frac{1}{{r}^{2k-1}} \right)^2 dv \\ 
&=\int_{0}^{2 \pi}  \cdots \int_{0}^{\pi} \int_{0}^{\pi} \left(\frac{1}{R_{1}^{2k-1}} - \frac{1}{\left( R_{2}^2 + x^2+ 2 \, R_{2} \, x \cos(\phi_1)\right) ^{\frac{2k-1}{2}}} \right)^2 \, R_{2}^{2k} \, d \phi \\
&= 4 \pi \, \frac{\pi}{2} \, \frac{4}{3} \cdots \left({\frac{2k-3}{2k-2}} \, {\frac{2k-5}{2k-4}} \, ....{\frac{\pi}{2}} \right) R_{2}^{2k} \int_{0}^{\pi} \left(\frac{1}{R_{1}^{2k-1}} - \frac{1}{\left( R_{2}^2 + x^2+ 2 \, R_{2} \, x \cos(\phi_1)\right) ^{\frac{2k-1}{2}}} \right)^2 \sin^{2k-1} (\phi_1) \, d \phi_1
\end{align*}
and
\begin{align*}
&\int_{\mathbb{S}(R_2)}\left(\frac{1}{R_{1}^{2k-1}} - \frac{1}{{R_{2}}^{2k-1}} \right)^2 dv \\
&= \int_{0}^{2 \pi} \int_{0}^{\pi} \int_{0}^{\pi} \cdots\int_{0}^{\pi} \left(\frac{1}{R_{1}^{2k-1}} - \frac{1}{ R_{2} ^{2k-1}} \right)^2 R_{2}^{2k} d \phi \\
&= 4 \pi  \, \frac{\pi}{2} \, \frac{4}{3} \cdots \left({\frac{2k-3}{2k-2}} \, {\frac{2k-5}{2k-4}} \, \cdots {\frac{\pi}{2}} \right) R_{2}^{2k} \left({\frac{2k-2}{2k-1}} \, {\frac{2k-4}{2k-3}} \, \cdots {\frac{4}{3}} \right) \left(\frac{1}{R_{1}^{2k-1}} - \frac{1}{ R_{2} ^{2k-1}} \right)^2.
\end{align*}
Thus \eqref{eqn:odd fun inequality} is equivalent to 
\begin{align}\nonumber
\label{eqn:odd fun inequality1}
\int_{0}^{\pi} \left(\frac{1}{R_{1}^{2k-1}} - \frac{1}{\left( R_{2}^2 + x^2+ 2 \, R_{2} \, x \cos(\phi_1)\right) ^{\frac{2k-1}{2}}} \right)^2 \sin^{2k-1} (\phi_1) d \phi_1 \\ 
\geq \left({\frac{2k-2}{2k-1}} \, {\frac{2k-4}{2k-3}} \, \cdots {\frac{4}{3}} \right) \left(\frac{1}{R_{1}^{2k-1}} - \frac{1}{ R ^{2k-1}} \right)^2.
\end{align}
Let $s = \cos (\phi_1)$, then the above inequality reduces to
\begin{align*}
\int_{-1}^{1} \left(\frac{1}{R_{1}^{2k-1}} - \frac{1}{\left( R_{2}^2 + x^2+ 2 \, R_{2} \, x \, s\right) ^{\frac{2k-1}{2}}} \right)^2 \left( 1-s^2\right)^{k-1}  d s \\
\geq \left({\frac{2k-2}{2k-1}} \, {\frac{2k-4}{2k-3}} \, \cdots {\frac{4}{3}} \right) \left(\frac{1}{R_{1}^{2k-1}} - \frac{1}{ R ^{2k-1}} \right)^2.
\end{align*}
This is equivalent to
\begin{align} \nonumber
\int_{-1}^{1} \left( \frac{1}{\left( R_{2}^2 + x^2+ 2 \, R_{2} \, x \, s\right) ^{2k-1}} - \frac{2}{ {R_{1}^{2k-1}} \left( R_{2}^2 + x^2+ 2 \, R_{2} \, x \, s\right) ^{\frac{2k-1}{2}}} \right) \left( 1-s^2\right)^{k-1} d s \\
\geq \left({\frac{2k-2}{2k-1}} \, {\frac{2k-4}{2k-3}} \, \cdots {\frac{4}{3}} \right) \left(\frac{1}{R_2^{4k-2}} - \frac{2}{ {{R_{1}^{2k-1}}}{R_{2} ^{2k-1}}} \right),
\end{align}
which follows from Lemma \ref{lem:odd1} and Lemma \ref{lem:odd2}. 
\end{proof} 
\begin{lemma} \label{lem:even1}
Let $x$ and $R_{2}$ be as above and $k \in \mathbb{N}$, $k \geq 2$. Then the following holds.
\begin{align}
\int_{-1}^{1} \left( \frac{\left( 1-s^2\right)^\frac{2k-3}{2}}{\left( R_{2}^2 + x^2+ 2 \, R_{2} \, x \, s\right) ^{2(k-1)}}   \right)  d s \geq
 \left({\frac{2k-3}{2k-2}} \, {\frac{2k-5}{2k-4}} \, \cdots {\frac{\pi}{2}} \right) \left(\frac{1}{R_2^{4k-4}}\right).
\end{align}
\end{lemma}
\begin{proof} By using MATHEMATICA, we get
\begin{align} \nonumber
\int_{-1}^{1} \left( \frac{\left( 1-s^2\right)^\frac{2k-3}{2}}{\left( R_{2}^2 + x^2+ 2 \, R_{2} \, x \, s\right) ^{2(k-1)}}   \right)  d s &= \frac{\sqrt{\pi}}{(R_{2}^2 + x^2)^{2k-2}} \, \frac{\Gamma\left( k - \frac{1}{2} \right)}{\Gamma\left( k \right)} \, {{_2\mbox{F}_1} \left(k -1,\, k - \frac{1}{2};\,k ;\,\frac{4 R_{2}^2 {x}^2}{(R_{2}^2 + x^2)^2} \right) }, \\  \label{eqn:lem even11}
&= {\sum_{t=0}^{\infty}} \hat{\alpha}(t) \frac{x^{2t}}{R_{2}^{2t+4k-4}},
\end{align}
where $\hat{\alpha}(t)$, $t \geq 0$ is given by
\begin{align}  \nonumber
\hat{\alpha}(t) & = \frac{ \sqrt{\pi} \, \Gamma(k - \frac{1}{2})}{ \Gamma\left( k \right)} \,  {\sum_{i=0}^{t}} \, \frac{ 4^{i} \, \Gamma \left( k +i - \frac{1}{2}\right) \,   \left(k - 1 \right) \, (-1)^{t-i} }{ \Gamma \left( k - \frac{1}{2} \right) \, \left(k - 1 +i \right) \, i! \, (t-i)!} \, \frac{\Gamma (i + 2k-2+t)}{\Gamma (2i + 2k-2)} \\ \label{eqn:lem even12}
&= \frac{ \sqrt{\pi} \, \Gamma(k - \frac{1}{2})}{ \Gamma\left( k \right)} \, \frac{(k-1)\, \Gamma ( 2k-2+t)}{\Gamma(2k-2) \, \Gamma (1+t) \, (k+t-1)}.
\end{align}
The sum of the above series is evaluated by using MATHEMATICA.

Observe that $R_2, x > 0$ and $\hat{\alpha}(t) > 0$ for all $t \geq 0$. Then it follows by \eqref{eqn:lem even11} and \eqref{eqn:lem even12} that
\begin{align*}
\int_{-1}^{1} \left( \frac{\left( 1-s^2\right)^\frac{2k-3}{2}}{\left( R_{2}^2 + x^2+ 2 \, R_{2} \, x \, s\right) ^{2(k-1)}}   \right)  d s & \geq \hat{\alpha}(0) \frac{1}{R_{2}^{4k-4}} \\
& = \left({\frac{2k-3}{2k-2}} \, {\frac{2k-5}{2k-4}} \, \cdots {\frac{\pi}{2}} \right) \, \frac{1}{R_2^{4k-4}}.
\end{align*}
This proves the lemma. 
\end{proof}

\begin{lemma} \label{lem:even2}
Let $x$ and $R_{2}$ be as above and $k \in \mathbb{N}$, $k \geq 2$. Then the following holds.
\begin{align}\label{eqn:lem even21}
\int_{-1}^{1} \left( \frac{\left( 1-s^2\right)^\frac{2k-3}{2}}{\left( R_{2}^2 + x^2+ 2 \, R_{2} \, x \, s\right) ^{k-1}}   \right)  d s = \left({\frac{2k-3}{2k-2}} \, {\frac{2k-5}{2k-4}} \, \cdots {\frac{\pi}{2}} \right) \left(\frac{1}{R_2^{2(k-1)}}\right).
\end{align}
\end{lemma}
\begin{proof} By proceeding the same way as in Lemma \ref{lem:even1}, we get
\begin{align} \nonumber
\int_{-1}^{1} \left( \frac{\left( 1-s^2\right)^\frac{2k-3}{2}}{\left( R_{2}^2 + x^2+ 2 \, R_{2} \, x \, s\right) ^{k-1}}   \right)  d s &= \frac{\sqrt{\pi}}{(R_{2}^2 + x^2)^{k-1}} \, \frac{\Gamma\left(k - \frac{1}{2}\right)}{\Gamma\left( k \right)} \, {{_2\mbox{F}_1} \left(\frac{k-1}{2},\frac{k}{2};\,k ;\,\frac{4 R_{2}^2 {x}^2}{(R_{2}^2 + x^2)^2} \right) } \\ \label{eqn:lem even22}
& = {\sum_{t=0}^{\infty}}\hat{\beta}(t) \frac{x^{2t}}{R_{2}^{2t+2k-2}},
\end{align}
where $\hat{\beta}(t)$ is given by
\begin{align} \nonumber
\hat{\beta}(t) & = \frac{ \sqrt{\pi} \, \Gamma(k - \frac{1}{2})}{ \Gamma\left( k \right)} {\sum_{i=0}^{t}}\, \frac{ 4^{i} \, \Gamma\left( \frac{k-1}{2} +i \right) \, \Gamma \left( \frac{k}{2} +i \right) \,  \Gamma \left(k \right) \, (-1)^{t-i} }{ \Gamma\left( \frac{k-1}{2} \right) \, \Gamma \left( \frac{k}{2} \right) \, \Gamma \left(k +i \right) \, i!    \, (t-i)!} \frac{\Gamma (i + k + t-1)}{\Gamma (2 i + k -1)} \\ \label{eqn:lem even23}
& =
\begin{cases}
\frac{ \sqrt{\pi} \, \Gamma(k - \frac{1}{2})}{ \Gamma\left( k \right)}, & \text{ for } t=0, \\
0, & \text{ for } t \neq 0.
\end{cases}
\end{align}
Hence the lemma follows by substituting the values of $\hat{\beta}(t)$ from \eqref{eqn:lem even23} into \eqref{eqn:lem even22}. 
\end{proof}

\begin{proposition} \label{prop: even fun inequality} Let $R_1$ and $R_2$ be defined as above and $r(p)$ be the distance of a point $p$ in $\mathbb{S}'(R_2)$ from the origin. For $n=2k, k \geq 2$ the following holds.
\begin{align}
\label{eqn:even fun inequality}
\int_{\mathbb{S}'(R_2)}\left(\frac{1}{R_{1}^{2(k-1)}} - \frac{1}{{r}^{2(k-1)}} \right)^2 dv \geq \int_{\mathbb{S}(R_2)}\left(\frac{1}{R_{1}^{2(k-1)}} - \frac{1}{{R_2}^{2(k-1)}} \right)^2 dv.
\end{align}
\end{proposition}
\begin{proof} By proceeding in the same way as in Proposition \ref{prop: odd fun inequality}, inequality \eqref{eqn:even fun inequality} is equivalent to
\begin{align*} \nonumber
\int_{-1}^{1} \left( \frac{1}{\left( R_{2}^2 + x^2+ 2 \, R_{2} \, x \, s\right) ^{2(k-1)}} - \frac{2}{ {R_{1}^{2(k-1)}} \left( R_{2}^2 + x^2+ 2 \, R_{2} \, x \, s\right) ^{k-1}} \right) \left( 1-s^2\right)^{\frac{2k-3}{2}} d s \\
\geq \left({\frac{2k-3}{2k-2}} \, {\frac{2k-5}{2k-4}} \, \cdots {\frac{\pi}{2}} \right) \left(\frac{1}{R_2^{4(k-1)}} - \frac{2}{ {{R_{1}^{2(k-1)}}}{R_{2} ^{2(k-1)}}} \right).
\end{align*}
This follows from Lemma \ref{lem:even1} and Lemma \ref{lem:even2}.  
\end{proof}

\subsection{Proof of Theorem \ref{thm:steklov}}
\begin{proof}
For $n>2$, since $\left(\frac{1}{R_{1}^{n-2}} - \frac{1}{{r}^{n-2}} \right) = 0 $ on $\mathbb{S}(R_{1})$, the function $\left(\frac{1}{R_{1}^{n-2}} - \frac{1}{{r}^{n-2}} \right)$ satisfies the condition for the test function in \eqref{var char for steklov}. Thus by considering $\left(\frac{1}{R_{1}^{n-2}} - \frac{1}{{r}^{n-2}} \right) $ as test function, we have 
\begin{align*}
\tau_{1} (D) \leq \frac{\displaystyle\int_{D}{\left\Vert\nabla \left( \frac{1}{R_{1}^{n-2}} - \frac{1}{r^{n-2}} \right) \right \Vert^2}dV}{\displaystyle\int_{\mathbb{S}'(R_2)}\left(\frac{1}{R_{1}^{n-2}} - \frac{1}{{r}^{n-2}} \right)^2 dv}.
\end{align*} 

By Proposition \ref{prop: odd fun inequality} and \ref{prop: even fun inequality}, we have
\begin{align} \label{steklov fun est}
\int_{\mathbb{S}'(R_2)}\left(\frac{1}{R_{1}^{n-2}} - \frac{1}{{r}^{n-2}} \right)^2 dv \geq \int_{\mathbb{S}(R_2)}\left(\frac{1}{R_{1}^{n-2}} - \frac{1}{{r}^{n-2}} \right)^2 dv,
\end{align}
for all $n > 2$.

Next we estimate $\displaystyle\int_{D}{\left\Vert\nabla \left( \frac{1}{R_{1}^{n-2}} - \frac{1}{r^{n-2}} \right) \right\Vert^2} dV$.
\begin{align*} 
\int_{D}{\left\Vert\nabla \left( \frac{1}{R_{1}^{n-2}} - \frac{1}{r^{n-2}} \right) \right\Vert^2} dV &= \int_{D} \left( \frac{n-2}{r^{n-1}}\right)^{2} dV \\ \nonumber
&=  \int_{D \cap \Omega_{0}} \left( \frac{n-2}{r^{n-1}}\right)^{2} dV + \int_{D\setminus{D\cap \Omega_{0}}} \left( \frac{n-2}{r^{n-1}}\right)^{2}dV. 
\end{align*}
 Since $r \geq R_{2}$  in $(D\setminus{D\cap \Omega_{0}})$,  we have
\begin{align*}
\int_{D}{\left\|\nabla \left( \frac{1}{R_{1}^{n-2}} - \frac{1}{r^{n-2}} \right) \right\|^2}dV & \leq \int_{\Omega_{0}} \left( \frac{n-2}{r^{n-1}}\right)^{2} dV  - \int_ {\Omega_{0} \setminus {D\cap \Omega_{0}}} \left( \frac{n-2}{r^{n-1}}\right)^{2} dV + \int_{D\setminus{D\cap \Omega_{0}}} \left( \frac{n-2}{R_{2}^{n-1}}\right)^{2}dV.
\end{align*}
Since Vol$(\Omega_{0} \setminus {D\cap \Omega_{0}})$ =  Vol$(D\setminus{D\cap \Omega_{0}})$, it follows that
\begin{align*}
\int_{D}{\left\|\nabla \left( \frac{1}{R_{1}^{n-2}} - \frac{1}{r^{n-2}} \right) \right\|^2} dV & \leq \int_{\Omega_{0}} \left( \frac{n-2}{r^{n-1}}\right)^{2} dV  + \int_ {\Omega_{0} \setminus {D\cap \Omega_{0}}} \left[ \left( \frac{n-2}{R_{2}^{n-1}}\right)^{2} - \left( \frac{n-2}{r^{n-1}}\right)^{2}\right]  dV.
\end{align*}
As $r \leq R_{2}$ in $(\Omega_{0} \setminus{D\cap \Omega_{0}})$, we get 
\begin{align} \nonumber
\int_{D}{\left\|\nabla \left( \frac{1}{R_{1}^{n-2}} - \frac{1}{r^{n-2}} \right) \right\|^2} dV & \leq \int_{\Omega_{0}} \left( \frac{n-2}{r^{n-1}}\right)^{2} dV \\ \label{steklov grad est.}
& = \int_{\Omega_{0}}{\left\|\nabla \left( \frac{1}{R_{1}^{n-2}} - \frac{1}{r^{n-2}} \right) \right\|^2} dV.
\end{align}
By inequality\eqref{steklov fun est} and \eqref{steklov grad est.}, we have
\begin{align*}
 \frac{\displaystyle\int_{D}{\left\|\nabla \left( \frac{1}{R_{1}^{n-2}} - \frac{1}{r^{n-2}} \right) \right\|^2} dV}{\displaystyle\int_{\mathbb{S}'(R_2)}\left(\frac{1}{R_{1}^{n-2}} - \frac{1}{{r}^{n-2}} \right)^2 dv} & \leq \frac{\displaystyle\int_{\Omega_{0}}{\left\|\nabla \left( \frac{1}{R_{1}^{n-2}} - \frac{1}{r^{n-2}} \right) \right\|^2}dV}{\displaystyle\int_{\mathbb{S}(R_2)}\left(\frac{1}{R_{1}^{n-2}} - \frac{1}{{r}^{n-2}} \right)^2 dv},  \\\\
 \tau_{1} (D) & \leq  \tau_{1} (\Omega_{0}).
\end{align*}
Further, equality holds if and only if Vol$(\Omega_{0} \setminus {D\cap \Omega_{0}}) = 0$. Since $\Omega_{0}$ and $D$ are open subsets of $\mathbb{R}^n$, Vol$(\Omega_{0} \setminus {D\cap \Omega_{0}}) = 0$ holds if and only if $D = \Omega_{0}$. 
Hence the theorem follows. 
\end{proof}

\begin{remark}
For $n=2$, we failed to prove above theorem since we don't have an analogue of \eqref{eqn:even fun inequality} if we consider $(\ln r - \ln R_{1})$ as a test function for the variational characterization of first eigenvalue.
\end{remark}
\section{Neumann eigenvalue problem}
We prove Theorem \ref{thm:neumann} in this section. Throughout this section, $(\mathbb{M}, ds^2)$ denotes a non-compact rank-$1$ symmetric space with dim $\mathbb{M} = m = kn$ where $k= \mbox{dim}_{\mathbb{R}} {\mathbb{K}}$; $\mathbb{K} = \mathbb{R}, \mathbb{C}, \mathbb{H}$ or $\mathbb{C}a$, and metric $ds^{2}$ is such that $-4 \leq K_{\mathbb{M}} \leq -1$. For a domain $\Omega$, the variational characterization for the first non-zero eigenvalue of \eqref{eqn:neumann} is given by
\begin{align} \label{var char for neumann}
\mu_1 = \inf \left\lbrace \frac{\int_{\Omega}{\|\nabla u\|^2} \,dv}{\int_{\Omega}{|u|^2} \, dv} \bigm| \, \int_{\Omega} u \, dv =0 \right\rbrace.
\end{align}

We start with the study of first non-constant eigenfunction of eigenvalue problem \eqref{eqn:neumann} on an annulus $\Omega_{0} \subset M$, where $M = \mathbb{M}$ or $\mathbb{R}^m$. Then we find test functions for a  geodesically symmetric doubly connected domain $\Omega \subset M$, and obtain an upper bound for the first non-zero eigenvalue of Neumann problem on $\Omega$. 

Let $p \in M$. If $\gamma(r) $ is a unit speed geodesic starting at $p$, then the Riemannian volume density at $\gamma (r)$ is 
\begin{align*}
J(r) =
\begin{cases}
 \sinh^{kn-1} r \,\cosh^{k-1} r & \text{ for } M = \mathbb{M}, \\
 r^{m-1} & \text{ for } M = \mathbb{R}^m,
\end{cases}
\end{align*}
and the first eigenvalue $\lambda_{1}(S(r))$ of the geodesic sphere $S(p,r)$ is 
\begin{align*}
\lambda_{1}(S(r))=
\begin{cases}
\left( \frac{kn-1}{\sinh^{2} r} - \frac{k-1}{\cosh^{2} r }\right) & \text{ for } M = \mathbb{M}, \\
\frac{m-1}{r^{2}} & \text{ for } M = \mathbb{R}^m.
\end{cases} 
\end{align*}
For a proof see \cite{ARGS}. 
\subsection{Neumann eigenvalue problem on an annulus}
Let $\Omega_{0} = \left\lbrace q \in M \mid r_1 < d(p,q) < r_2 \right\rbrace  $ be an annular domain centered at point $p \in M$. Consider the following problem
\begin{align} \label{eq2l}
\begin{array}{rcll}
\Delta u &=& \mu \ u \, &\mbox{ in } \, \Omega_{0}, \\
\frac{\partial u}{\partial \nu} &=& 0 \,  &\mbox{ on } \, {\partial \Omega_{0}},
\end{array}
\end{align}
where $\nu$ is the outward unit normal on $\partial \Omega_{0} $.
We will study the first non-zero eigenvalue $\mu_1(\Omega_{0}) $ of \eqref{eq2l}.

The first non-zero eigenvalue of \eqref{eq2l} is, by separation of variables technique, either the second eigenvalue $\tau_2$ of
\begin{align}\label{eq22}
- \frac{1}{J(r)} \frac{\partial}{\partial r} \left( J(r) \frac{\partial}{\partial r} f \right) = \tau f,
\end{align}
where $f$ is defined on $[r_1,r_2]$ with $f'(r_1)=0= f'(r_2)$ or the first eigenvalue $\mu_1$ of
\begin{align}\label{eq23}
- \frac{1}{J(r)} \frac{\partial}{\partial r} \left( J(r) \frac{\partial}{\partial r} g \right) + \lambda_1(S(r)) g = \mu g,
\end{align}
where $\lambda_1(S(r))$ is the first non-zero eigenvalue of Laplacian on $S(r)$ corresponding to eigenfunctions $\frac{x_i}{r}$, where $(x_1, x_2,\ldots, x_m)$ is the normal coordinate system centered at $p$ and the function $g$ is defined on $[r_1,r_2]$ with $g'(r_1)=0= g'(r_2)$.

Let $g$ be an eigenfunction corresponding to the first eigenvalue $\mu_1$ of \eqref{eq23}. Then $g$ does not change sign in $(r_1,r_2)$. We may assume that $g$ is positive on $(r_1,r_2)$.
Let $f$ be an eigenfunction corresponding to the second eigenvalue $\tau_2$ of \eqref{eq22}. Then the function $f$ must change sign in $(r_1,r_2)$. Let $a \in (r_1,r_2) $ be such that $f$ is positive on interval $(r_1,a)$ and negative on $(a,r_2)$. Then $f'(a)<0$. 
Let $h$ be a non-trivial solution of
\begin{align}\label{eq24}
- \frac{1}{J(r)} \frac{\partial}{\partial r} \left( J(r) \frac{\partial}{\partial r} h \right) = \mu_1 h .
\end{align}
By differentiating above equation, we get 
\begin{align*}
- \frac{1}{J(r)} \frac{\partial}{\partial r} \left( J(r) \frac{\partial}{\partial r} h' \right) + \lambda_1(S(r)) h' = \mu_1 h'.
\end{align*}
Since $h'$ and $g$ satisfy \eqref{eq23} with the same eigenvalue, we may assume that $h'=g$. Using \eqref{eq24} and the fact that $f$ satisfies \eqref{eq22} with eigenvalue $\tau_2$, we get
\begin{align}\label{eq25}
\frac{\partial}{\partial r} \left( J(r) \left( f h'-h f'\right) \right) = \left( \tau_2 - \mu_1 \right) f h J(r).
\end{align}
By integrating above equation and using the values of $f$, we obtain
\begin{align}\label{eq26}
\begin{split}
\int_{r_1}^{a} \left( \tau_2 - \mu_1 \right) f h J(r) &= \left[  J(r) \left( f h'-h f'\right) \right]_{r_1}^{a} \\
&= -J(a) h(a) f'(a) - J(r_1) f(r_1) h'(r_1).
\end{split}
\end{align}
Since
\begin{align*}
\mu_1 h(r_2) = -g'(r_2)- \frac{J'(r_2)}{J(r_2)} g(r_2), \, g'(r_2)=0
\end{align*}
and $g$ is positive in the interval $(r_1,r_2)$, we have $ h(r_2)<0$. As $h'=g$ and $g$ is a positive function, it follows that $h$ is an increasing function. So $h\leq 0$ in $(r_1,r_2)$. Then from \eqref{eq26} it follows that $\tau_2 > \mu_1$. Thus $\mu_1= \mu_1(\Omega_{0})$. 

\begin{remark}
Let $g(r)$ be the solution of \eqref{eq23} and $(x_1, x_2,\ldots, x_m)$ be the geodesic normal coordinates centered at $p$.  Then the function $g(r) \frac{x_i}{r}$, $1 \leq i \leq m$  are eigenfunctions of \eqref{eq2l} corresponding to the first non-zero eigenvalue.
\end{remark}

We now prove that $g$ is an increasing function on $(r_1,r_2)$.
\begin{lemma}
$g'(r)>0$ on $(r_1,r_2)$ and $\mu_1(\Omega_{0})> \lambda_1(S(r_2))$.
\end{lemma}

\begin{proof} The function $g$ satisfies the equation
\begin{align}\label{eq27}
\frac{\partial}{\partial r} \left( J(r) \frac{\partial}{\partial r} g \right) = \left( \lambda_1 (S(r))- \mu_1(\Omega_{0})\right) g(r) J(r)
\end{align}
with $g'(r_1)=0= g'(r_2)$. Define $\Psi(r)= J(r) g'(r)$. Then $\Psi(r_1)=0= \Psi(r_2)$. Since $\Psi $ can not be identically zero, $\Psi '(r)$ must change sign atleast once in the interval $(r_1,r_2)$. Hence $\left( \lambda_1 (S(r))- \mu_1(\Omega_{0})\right)$ must change sign atleast once in the interval $(r_1,r_2)$. Since $ \lambda_1 (S(r))$ is a strictly decreasing function, $\left( \lambda_1 (S(r_{1}))- \mu_1(\Omega_{0})\right)<0 $ implies $\left( \lambda_1 (S(r))- \mu_1(\Omega_{0})\right)<0 \mbox{ for all } r \in (r_1,r_2)$. This contradicts to the fact that $\left( \lambda_1 (S(r))- \mu_1(\Omega_{0})\right)$ change sign atleast once in the interval $(r_1,r_2)$. Thus $\left( \lambda_1 (S(r))- \mu_1(\Omega_{0})\right)$ must be positive on an interval  $(r_1,b)$ for some point $b \in (r_1,r_2)$ and negative on $(b,r_2)$. Hence $\mu_1(\Omega_{0})> \lambda_1(S(r_2))$. 
 
 Further, $\Psi '(r)$ is positive on the interval $(r_1,b)$ and negative on $(b,r_2)$. Since $\Psi$ is an increasing function in $(r_1,b)$ and a decreasing function in $(b,r_2)$, it follows that $\Psi(r)$ is positive on $(r_1,r_2)$. Hence $g'(r)>0$ on $(r_1,r_2)$. 
\end{proof} 

\subsection{Neumann eigenvalue problem on geodesically symmetric domain in $\mathbb{R}^m$ and non-compact Rank-1 symmetric spaces} 

We prove Theorem \ref{thm:neumann} in this subsection. We begin with the following lemma, which is crucial to prove the theorem.

Let $g$ be the first non-constant eigenfunction of \eqref{eq23} on $[r_1,r_2]$. Let $
B(r)= (g'(r))^2 + \lambda_{1}(S(r)) \, g^{2}(r)
$.

\begin{lemma} \label{lem: for function b(r)}
$B'(r) \leq 0$ for $r \in [r_1,r_2]$.
\end{lemma}
\begin{proof} We prove lemma for $M= \mathbb{R}^m$ and $M= \mathbb{M}$. For $M = \mathbb{R}^m$, $
B(r)$ takes the form $(g'(r))^2 + \frac{m-1}{r^2} \, g^{2}(r)
$. Then
\begin{align*}
B'(r) &= 2 \, g'(r)  \left[  \left( \frac{m-1}{r^2}- \mu_1(\Omega_{0})\right) g(r) - \frac{m-1}{r} g'(r)\right]  + 2 \, g(r) \, g'(r) \left( \frac{m-1}{r^{2} }\right) -2 \, g^{2}(r)  \left( \frac{m-1}{r^{3} }\right) \\
& \leq 4 \,\left( \frac{m-1}{r^{2} }\right) g(r) \, g'(r)  -2 \,\left( \frac{m-1}{r^{3} }\right) g^{2}(r)   -2 \, \left( \frac{m-1}{r}\right)  g'^{2}(r) \\
&\leq 0.
\end{align*}

For $M = \mathbb{M}$, $
B(r)$ takes the form $ (g'(r))^2 + \left( \frac{kn-1}{\sinh^{2} r} - \frac{k-1}{\cosh^{2} r }\right) \, g^{2}(r)$. Then
\begin{align*}
B'(r) &= 2 \, g'(r)  \left( \left( \lambda_1 (S(r))- \mu_1(\Omega_{0})\right) g(r) - \frac{J'(r)}{J(r)} g'(r)\right) + 2 \, g(r) \, g'(r) \left( \frac{kn-1}{\sinh^{2} r} - \frac{k-1}{\cosh^{2} r }\right)\\ 
&\quad -2 \, g^{2}(r)  \left( \frac{(kn-1)\cosh r}{\sinh^{3} r} - \frac{(k-1)\sinh r}{\cosh^{3} r }\right) \\
&= 4 \, g(r) \, g'(r) \left( \frac{kn-1}{\sinh^{2} r} - \frac{k-1}{\cosh^{2} r }\right) - 2 \, g(r) \, g'(r) \, \mu_1(\Omega_{0}) \\
&\quad - 2 \, (g'(r))^2 \left( \frac{(kn-1) \cosh r}{\sinh r} + \frac{(k-1)\sinh r}{\cosh r}\right) -2 \, g^{2}(r) \left( \frac{(kn-1)\cosh r}{\sinh^{3} r} - \frac{(k-1)\sinh r}{\cosh^{3} r }\right)\\
& \leq \frac{2(kn-1)}{\sinh^3 r }\left( 2 \, g(r) \, g'(r) \, \sinh r - (g'(r))^2 \, \cosh r \, \sinh^{2} r - g^{2}(r) \, \cosh r \right)  \\
&\quad - \frac{2(k-1)}{\cosh^3 r } \left( 2 \, g(r) \, g'(r) \, \cosh r + (g'(r))^2 \, \sinh r \, \cosh^{2} r - g^{2}(r) \, \sinh r \right).
\end{align*}
Since $\cosh r \geq 1$, it follows that
\begin{align*}
B'(r) &\leq \frac{2k(n-1)}{\sinh^3 r }\left( 2 \, g(r) \, g'(r) \, \sinh r - (g'(r))^2  \, \sinh^{2} r - g^{2}(r) \right) \\
&\quad + \frac{2(k-1)}{\sinh^3 r }\left( 2 \, g(r) \, g'(r) \, \sinh r - (g'(r))^2 \, \cosh r \, \sinh^{2} r - g^{2}(r) \, \cosh r \right)                                                             \\
&\quad - \frac{2(k-1)}{\cosh^3 r } \left( 2 \, g(r) \, g'(r) \, \cosh r + (g'(r))^2 \, \sinh r \, \cosh^{2} r - g^{2}(r) \, \sinh r \right)\\
& = - \frac{2k(n-1)}{\sinh^3 r } \left(g'(r)  \, \sinh r - g(r) \right)^{2} - 2 \, (k-1) \left( 2 \, g(r) \, g'(r)\left( \frac{1}{\cosh^{2} r} - \frac{1}{\sinh^{2} r}\right)\right. \\
 &\quad \left. + (g'(r))^2 \left( \frac{\cosh r}{\sinh r} + \frac{\sinh r}{\cosh r}\right) +g^{2}(r)\left(\frac{\cosh r}{\sinh^3 r} - \frac{\sinh r}{\cosh^3 r} \right)  \right) \\
& \leq - 2 \, (k-1) \left( \frac{- 8 \, g(r) \, g'(r)}{\sinh^{2} {2r}} +  (g'(r))^2 \, \frac{2 \, \cosh {2r}}{\sinh {2r}} + g^{2}(r)\, \frac{8 \,  \cosh {2r} }{\sinh^{3} {2r}} \right) \\
& \leq \frac{-4 \, (k-1)}{\sinh^{3} {2r}} \left( 2 \, g(r) - g'(r) \, \sinh {2r}\right)^{2} \\
& \leq 0.         
\end{align*}
This proves the lemma. 
\end{proof}
\vspace{0.1 cm}

\begin{proof}[Proof of Theorem \ref{thm:neumann}]
Let $D$ be a domain in $M$ which is geodesically symmetric with respect to a point $p \in M$ and $B_0$ be a ball of radius $r_1$ contained in $D$ with centre at $p$. Let $B(r_2)$ be a ball in $M$ of radius $r_2$, centered at $p$ such that $\text{Vol}(D)= \text{ Vol}(B(r_2))$. Define $\Omega_{0} = (B(r_2)) \setminus \overline{B_0} $. 

Define the function $h$ on $\Omega = D \setminus \overline{B_0}$ by
$$
h(x)=
\begin{cases}
g(\Vert x \Vert),  & \text{for } \Vert x \Vert\leq r_2 \\
g(r_2),  & \text{for } \Vert x \Vert \geq r_2 .\\
\end{cases}
$$
We observe that $h(x)$ is a radial function and denote it by $h(r)$. 
Since domain $D$ is geodesically symmetric with respect to the point $p$, it follows that 
$
\int_{\Omega} h(r) \, \frac{x_i}{r} \, dv = 0,
$
where $(x_1, x_2, \ldots,x_m)$ is the normal coordinate system centered at point $p$. 
Since $h(r) \, \frac{x_i}{r}$ satisfies the property of the test function in \eqref{var char for neumann}, the variational formulation of $\mu_1(\Omega)$ gives
\begin{align} \nonumber
\mu_1(\Omega) \, {\sum_{i=1}^{m}}  \int_{\Omega} \left( h(r) \frac{x_i}{r} \right)^2 \, dv & \leq {\sum_{i=1}^{m}} \int_{\Omega} \left\Vert {\nabla  \left( h(r) \frac{x_i}{r} \right) }  \right\Vert^{2} dv, \\ \label{variation char neumann}
\mu_1(\Omega) \int_{\Omega} h^{2}(r) \, dv  & \leq {\sum_{i=1}^{m}} \int_{\Omega} \left\Vert \nabla  \left( h(r) \frac{x_i}{r} \right)   \right\Vert^{2} dv.
\end{align} 

Next we estimate $\int_{\Omega}  h^{2}(r) dv$ and ${\sum_{i=1}^{m}} \int_{\Omega} \Vert \nabla  \left( h(r) \frac{x_i}{r} \right)   \Vert^{2} dv$. \\

\textbf{Estimate for $\int_{\Omega}  h^{2}(r) dv$:}
\begin{align*} 
\int_{\Omega}  h^{2}(r) dv &= \int_{\Omega_{0} \cap \Omega}  h^{2}(r) dv + \int_{\Omega \setminus (\Omega_{0} \cap \Omega)}  h^{2}(r) dv \\ 
& = \int_{\Omega_{0}}  h^{2}(r) dv - \int_{\Omega_{0} \setminus (\Omega_{0} \cap \Omega)}  h^{2}(r) dv + \int_{\Omega \setminus (\Omega_{0} \cap \Omega)}  h^{2}(r) dv \\ 
& = \int_{\Omega_{0}}  g^{2}(r) dv - \int_{\Omega_{0} \setminus (\Omega_{0} \cap \Omega)} g^{2}(r) dv + \int_{\Omega \setminus (\Omega_{0} \cap \Omega)}  g^{2}(r_2)  dv.
\end{align*}
Since Vol$(\Omega_{0} \setminus {\Omega_{0}\cap \Omega})$ =  Vol$(\Omega\setminus{\Omega_{0}\cap \Omega})$, it follows that
\begin{align} \nonumber
\int_{\Omega}  h^{2}(r) dv& = \int_{\Omega_{0}}  g^{2}(r) dv + \int_{\Omega_{0} \setminus (\Omega_{0} \cap \Omega)} \left( g^{2}(r_2) - g^{2}(r) \right) dv \\ \label{est fun neumann}
& \geq \int_{\Omega_{0}}  g^{2}(r) dv,
\end{align}
where the last inequality follows from the fact that $g$ is an increasing function on $[r_1, r_2]$. \\

\textbf{Estimate for ${\sum_{i=1}^{m}} \int_{\Omega} \Vert \nabla  \left( h(r) \frac{x_i}{r} \right)   \Vert^{2} dv$:} 

Note that
\begin{align} \label{est. grad neumann1}
{\sum_{i=1}^{m}} \int_{\Omega} \left\Vert \nabla  \left( h(r) \frac{x_i}{r} \right)   \right\Vert^{2} dv =  \int_{\Omega}\left( (h'(r))^2 + \lambda_{1}(S(r)) h^{2} (r) \right) dv .
\end{align} 
Let $B(r) = (h'(r))^2 + \lambda_{1}(S(r)) h^{2} (r)$. Note that $B(r)$ is a decreasing function on $[r_1, r_2].$ For $r \geq r_2$, $h(r)$ is a constant function and $\lambda_{1} S(r)$ is a decreasing function of $r$. This implies that $B(r)$ is a decreasing function on $r \geq r_2$. Thus \eqref{est. grad neumann1} can be written as
\begin{align} \nonumber
\int_{\Omega} B(r) dv & =  \int_{\Omega_{0}}  B(r) dv- \int_{\Omega_{0} \setminus (\Omega_{0} \cap \Omega)}  B(r) dv + \int_{\Omega \setminus (\Omega_{0} \cap \Omega)}  B(r) dv.
\end{align}
As $\, B(r) \leq B(r_2)$ on  ${\Omega \setminus (\Omega_{0} \cap \Omega)}$ and Vol$(\Omega_{0} \setminus {\Omega_{0}\cap \Omega})$ =  Vol$(\Omega\setminus{\Omega_{0}\cap \Omega})$, we get 
\begin{align*}
\int_{\Omega} B(r) dv & \leq \int_{\Omega_{0}}  B(r) dv + \int_{\Omega_{0} \setminus (\Omega_{0} \cap \Omega)}  (B(r_2) - B(r)) dv. 
\end{align*}
Since  $B(r) \geq B(r_2)$ on  ${\Omega_{0} \setminus (\Omega_{0} \cap \Omega)}$, it follows that
\begin{align*}
\int_{\Omega} B(r) dv& \leq \int_{\Omega_{0}}  B(r) dv.
\end{align*}
By substituting above values in \eqref{variation char neumann}, we have
\begin{align*}
\mu_1(\Omega)   & \leq  \frac{{\sum_{i=1}^{m}} \int_{\Omega_{0}} \Vert \nabla  \left( g(r) \frac{x_i}{r} \right)   \Vert^{2} dv}{ \int_{\Omega_{0}}  g^{2}(r) dv} =  \mu_1(\Omega_{0}).
\end{align*}
Further, equality holds if and only if Vol$(\Omega_{0} \setminus {\Omega_{0}\cap \Omega}) = 0$, which is true if and only if $\Omega = \Omega_{0}$. 
\end{proof}

\begin{remark}
In the case of compact rank-$1$ symmetric space, we don't have an analogue of Lemma \ref{lem: for function b(r)}. Hence the above proof fails for the compact rank-$1$ symmetric spaces.
\end{remark}

\section*{Acknowledgement}
The first author is supported by UGC, government of India(Award letter no- 2-62/98(SA-I)).






\begin{thebibliography}{20}


\bibitem{ARRR}
        {\sc A. R. Aithal and R. Raut},
        {\it On the extrema of Dirichlet's first eigenvalue of a family of punctured regular polygons in two dimensional space forms},
        Proc. Indian Acad. Sci. (Math. Sci.), {\bf 122}, 2 (2012), 257--281.        



\bibitem{ARGS}
        {\sc A. R. Aithal and G. Santhanam},
        {\it Sharp upper bound for the first Neumann eigenvalue for bounded domains in rank-$1$ symmetric spaces},
        Trans. Amer. Math. Soc., {\bf 348}, 10 (1996), 3955--3965.        

\bibitem{RTIB}
        {\sc R. Banuelos, T. Kulczycki, I. Polterovich and B. Siudeja},
        {\it Eigenvalue  inequalities  for mixed Steklov problems, in Operator Theory and Its Applications},
        Amer. Math. Soc. Transl.Ser. 2, {\bf 231} (2010), 19-34.
  
\bibitem{ACMV}
        {\sc Anisa M. H. Chorwadwala and M. K. Vemuri},
        {\it Two functionals connected to the Laplacian in a class of doubly connected domains on rank one symmetric spaces of non-compact type},
        Geom Dedicata, {\bf 167} (2013), 11--21.        
         
 \bibitem{ACAR}
        {\sc Anisa M. H. Chorwadwala and A. R. Aithal},
        {\it On two functionals connected to the Laplacian in a class of doubly connected domains in space-forms},
        Proc. Indian Acad. Sci. Math. Sci., {\bf 115}, 1 (2005), 93--102.  
      
  \bibitem{SK1}
        {\sc A. El Soufi and R. Kiwan},
        {\it Extremal first Dirichlet eigenvalue of doubly connected plane domains and dihedral symmetry},
        SIAM Journal on Mathematical Analysis, {\bf 39}, 4 (2007), 1112--1119.    
      
 \bibitem{SK2}
        {\sc A. El Soufi and R. Kiwan},
        {\it Where to place a spherical obstacle so as to maximize the second Dirichlet eigenvalue},
        Communications  on  Pure  and  Applied Analysis, {\bf 7}, 5 (2008), 1193--1201.    
            
 \bibitem{SK}
        {\sc S. Kesavan},
        {\it On two functionals connected to the Laplacian in a class of doubly connected domains},
        Proc. R. Soc. Edinb. Sect. A, {\bf 133}, 3 (2003), 617--624.
        
 \bibitem{ARPS}
        {\sc A. G. Ramm and P. N. Shivakumar},
        {\it Inequalities for the minimal eigenvalue of the Laplacian in an annulus},
        Math. Inequal. Appl., {\bf 1}, 4 (1998), 559--563.
        
 \bibitem{matica}Wolfram Research, Inc., Mathematica, Version 7.0, Champaign, IL (2008).
\end{thebibliography}
\end{document}